\documentclass[12pt,twoside]{amsart}
\usepackage[all]{xy}
        \usepackage {amssymb,latexsym,amsthm,amsmath,mathtools,dsfont}
        \usepackage{enumitem,color}
        \usepackage[mathscr]{euscript}

\usepackage{mathtools}

\long\def\symbolfootnote[#1]#2{\begingroup%
\def\thefootnote{\fnsymbol{footnote}}\footnote[#1]{#2}\endgroup}
\newcommand{\D}{\ensuremath{\mathcal{D}}}

\newcommand{\Z}{\ensuremath{\mathbb{Z}}}
\newcommand{\g}{\textup{Gal}}
\newcommand{\ha}{\textup{Hom}}
\newcommand{\h}{\textup{Hol}}
\newcommand{\p}{\textup{Perm}}
\newcommand{\au}{\textup{Aut}}

\newcommand{\ra}{\mathfrak{R}}

\newcommand{\ga}{\mathfrak{g}}
\newcommand{\fa}{\mathfrak{f}}
\newcommand{\ma}{\text{Map}}

\makeatletter
\def\imod#1{\allowbreak\mkern10mu({\operator@font mod}\,\,#1)}
\makeatother
\makeatletter
\renewcommand*\env@matrix[1][*\c@MaxMatrixCols c]{%
  \hskip -\arraycolsep
  \let\@ifnextchar\new@ifnextchar
  \array{#1}}
\makeatother

\newtheorem{theorem}{Theorem}[section]

\newtheorem{corollary}[theorem]{Corollary}
\newtheorem{proposition}[theorem]{Proposition}
\newtheorem*{theorem*}{Theorem}
\theoremstyle{definition}
\newtheorem{definition}[theorem]{Definition}
\newtheorem{remark}[theorem]{Remark}
\newtheorem{example}[theorem]{Example}
\numberwithin{equation}{section}

\newcommand{\ignore}[1]{}

\newcommand{\mynote}[1]{}
\begin{document}

\setcounter{section}{0}
\title{Hopf-Galois Realizability of $\mathbb{Z}_n\rtimes\mathbb{Z}_2$}
\author{Namrata Arvind and Saikat Panja}
\email{namchey@gmail.com, panjasaikat300@gmail.com}
\address{IISER Pune, Dr. Homi Bhabha Road, Pashan, Pune 411 008, India}
\thanks{The first named author is partially supported by the IISER Pune research fellowship and the second author has been partially supported by supported by NBHM fellowship.}
\date{\today}
\setcounter{tocdepth}{4}

\begin{abstract}
Let $G$ and $N$ be finite groups of order $2n$ where $n$ is 
odd. 
We say the pair $(G,N)$ is Hopf-Galois realizable if $G$ is a regular subgroup of $\h(N)=N\rtimes\au(N)$. In this article we give necessary conditions on $G$ (similarly $N$) when $N$ (similarly $G$) is a group of the form $\mathbb{Z}_n\rtimes\mathbb{Z}_2$.
Further we show that this condition is also sufficient if radical of $n$ is a Burnside number.
This classifies all the skew braces which has the additive group (or the multiplicative group) to be isomorphic to $\mathbb{Z}_n\rtimes\mathbb{Z}_2$, in this case.
\end{abstract}
\maketitle
\section{Introduction}
\subsection{Hopf-Galois structures} Let $K/F$ be a finite Galois field extension. An $F$-Hopf algebra $\mathcal{H}$, with an action on $K$ such that $K$ is an $H$-module algebra
and the action makes $K$ into an $\mathcal{H}$-Galois extension, will be called a \textit{Hopf-Galois structure}
on $K/F$.
From \cite[Theorem 6.8]{c}, we have that if $K/F$ be a Galois extension of fields and $G=\g(K/F)$, then there is a bijection between Hopf-Galois structures on $K/F$ and regular subgroups $N$ of $\p(G)$ normalized by $\lambda(G)$ where $\lambda$ is the left regular representation.
 In the proof of this theorem, given a regular subgroup $N\leq \p (G)$ normalized by $\lambda(G)$, the Hopf-Galois structure on $K/F$ corresponding to $N$ is $K[N]^{G}$. 
Here $G$ acts on $N$ by conjugation inside $\p(G)$ and it acts on $K$ by field automorphism, which induces an action of $G$ on $K[N]$. This was further simplified by Nigel Byott in 
{\cite[Proposition 1]{b}}, which states that this is equivalent to having $G$ as a regular subgroup of $\h(N)=N\rtimes\au(N)$. We are now ready to define Hopf-Galois realizability.
\begin{definition}
Let $G,N$ be two finite groups such that $|G|=|N|$. Then the pair $(G,N)$ is called Hopf-Galois realizable if $G$ is a regular subgroup of $\h(N)$.
\end{definition}
\begin{example}
Let $G=\mathbb{Z}_{2k}$ be the cyclic group of order $2k$ and $N=\D_{2k}$ be the dihedral group of order $2k$. Then the pair $(G,N)$ is Hopf-Galois realizable.
\end{example}
\subsection{Skew-braces}
\begin{definition}
A left skew brace is a triple $(\Gamma,+,\circ)$ where $(\Gamma,+),(\Gamma,\circ)$ are groups and satisfy
\[a\circ(b+c)=(a\circ b)+a^{-1}+(a\circ c),\] for all $a,b,c\in\Gamma$. The groups $(\Gamma,+),(\Gamma,\circ)$ will be called the additive group and the multiplicative group respectively. 
\end{definition}
\begin{definition}
Let $G,N$ be two finite groups such that $|G|=|N|$. Then the pair $(G,N)$ is called skew brace realizable if the exists a skew brace 
$(\Gamma,+,\circ)$ such that $(\Gamma,+)\cong N$ and $(\Gamma,\circ)\cong G$.
\end{definition}

Now given a skew brace $(\Gamma,+,\circ)$, we have two left regular representations:
\begin{align*}
    &\lambda_+ : \Gamma \longrightarrow \p(\Gamma)\\
    \text{and }&\lambda_{\circ} : \Gamma \longrightarrow \p(\Gamma)
\end{align*}
where $\lambda_+ (g)(x) = g+x$ and $\lambda_{\circ} (g)(x) = g \circ x $.
We will need 
the following result which connects skew braces and regular subgroups.
\begin{proposition}\cite[Proposition 1.2]{c1}
$(\Gamma,+,\circ)$ is a skew brace if and only if the homomorphism \begin{equation*}
    \lambda_{\circ} : (\Gamma,\circ) \longrightarrow \p(\Gamma)
\end{equation*} has image in $\h(\Gamma, +)$.

\end{proposition}
\begin{remark}
The above proposition tells us that a pair of groups $(G,H)$ (of the same order) is Hopf-Galois realizable if and only if it is skew brace realizable.
Thus from now on we will use the phrase "realizability of a pair of groups" instead of Hopf-Galois (or skew brace) realizability.
\end{remark}

\subsection{Results}
It is known from \cite{b1} that if $G$ is a finite simple group and $N$ is finite group of same order as $G$, then the pair $(G,N) $ is realizable if and only if $G\cong N$. This result was extended to a quasisimple groups in \cite{t2}. Also in \cite{ccd} all Hopf-Galois structures on groups of order $p^2q$, with cyclic sylow $p$ subgroup were classified. In \cite{r}, the author gives a necessary condition on $G$ if $(G,N)$ is realizable for a cyclic $N$.
Fixing $G$ to be cyclic, in \cite{t} a complete characterization of $N$ is given whenever $(G,N)$ is realizable.

Let $n \in \mathbb{Z}$ be odd. In this article we first fix $G$ to 
be a group of the form $\mathbb{Z}_n\rtimes\mathbb{Z}_2$ and 
classify all such $N$ for which $(G,N)$ is realizable. Turns out in 
this case $N \cong H_n\rtimes \mathbb{Z}_2$, where $H_n$ is a 
C-group (see section 2 for a definition of a C-group). Next we fix $N$ to be a group of the form 
$\Z_n\rtimes\Z_2$ and show that 
$G\cong H_{n}^{'}\rtimes\Z_2$ for a C-group $H_{n}^{'}$, whenever $(G,N)$ is 
realizable. Further the converse of the above statements is also 
true if radical of $n$ is assumed to be a Burnside number. The 
number of Hopf-Galois structures in these cases has been calculated in \cite{ap}.

We prove a group theoretic result and collect few known results in first part of section $2$. 
In the second part, we mention the results for realizability of pairs $(G,N)$ where one of the group is a Cyclic. Section $3$ is concerned with proof of all the main results.
\section{Preliminaries}

\subsection{Group Theoretic Results}
\begin{proposition}\label{p001}
Let $G$ be a groups of order $2n$ such that $n$ is odd. Then $G$ has a unique subgroup of order $n$.
\end{proposition}
\begin{proof}
Let $g\in G$ be an element of order $2$ (this exists by Cauchy's theorem). Consider the following compositions of the maps:
\begin{align*}
    G\xrightarrow{\lambda}\p(G)\xrightarrow{\text{sgn}}\{\pm1\},
\end{align*}
where $\lambda$ represents the left regular representation and $\text{sgn}$ is the sign representation of symmetric group. 
Note that $\lambda(g)=(a_1,ga_1)(a_2,ga_2)\ldots(a_n,ga_n)$ for $a_1,a_2,\ldots,a_n\in G$ such that 
$g^ka_l\neq a_m$ for all $k,l,m$. Since $n$ is odd, we have that $\text{sgn}(\lambda(g))=-1$ and hence $\text{sgn}\circ \lambda$ is surjective. Thus 
$H=\ker(\text{sgn}\circ \lambda)$ is a subgroup of order $n$.

To prove that $H$ is unique subgroup of order $n$, on contrary assume that there is another subgroup $H'\neq H$ of order $n$. Then consider the following group homomorphism
\begin{align*}
\varphi:G\longrightarrow(G/H)\oplus(G/H')\cong\Z_2\oplus\Z_2.    
\end{align*}
Since $H\neq H'$, we have that $G$ surjects into $\Z_2\times\Z_2$. Then $4$ divides the order of the group, which is a contradiction. 
\end{proof}
\begin{corollary}\label{c001}
Let $G$ be a group of order $2^kn$ such that $2\not|n$. If the Sylow-2-subgroup of $G$ is cyclic then $G$ has unique subgroups of order $2^ln$ for all $0\leq l\leq k$.
\end{corollary}
\begin{proof}
Follows by induction, Proposition \ref{p001} and observing that subgroups of cyclic groups are cyclic.
\end{proof}
\begin{definition}
A finite group $G$ is called a $C-group$ if all the Sylow subgroups are cyclic. The group $G$ is called \textit{almost Sylow-cyclic} if
its Sylow subgroups of odd order are cyclic, while either the Sylow-2-subgroup is trivial or they contain a cyclic subgroup of index $2$.
\end{definition}
We have the following theorem from Burnside's work.
\begin{proposition}\cite{bu}
Let $G$ be a finite group. Then all the Sylow subgroups are cyclic if and only if $G$ is a semidirect product of two cyclic groups of coprime order.
\end{proposition}
\subsection{Realizability results in case of Cyclic groups}
The notion of bijective crossed homomorphism to study Hopf-Galois module was introduced by Cindy Tsang in the paper \cite{t1}.
For a map $\fa\in\ha(G,\au(N))$, a map $\ga\in \ma(G,N)$ is said to be a crossed homomorphism with respect to $\fa$ if 
\begin{align*}
    \ga(ab)=\ga(a)\fa(a)(\ga(b))\text{ for all }a,b\in G.
\end{align*}
We set
\begin{align*}
    Z_{\fa}^1(G,N)=\{\ga:\ga \text{ is bijective crossed homomorphism w.r.t }\fa\}.
\end{align*}
Then we have the following.
\begin{proposition}\cite[Proposition 2.1]{t1}\label{p002}
The regular subgroups of $\h(N)$ isomorphic to $G$ are precisely the subsets of $\h(N)$ of the form
\begin{align*}
    \{(\ga(a),\fa(a)):a\in G\},
\end{align*}
where $\fa\in\ha(G,\au(N)),\ga\in Z_{\fa}^1(G,N)$.
\end{proposition}
Using this it has been proved that if $(G,N)$ is realizable, then
\begin{enumerate}
    \item if $G$ is simple then $N$ is simple,
    \item if $G$ is quasisimple then $N$ is quasisimple,
    \item if $G$ is almost simple then $N$ is almost simple.
\end{enumerate}
To study the realizability of pair of groups $(G,N)$ it is often useful to use charcteristic subgroups of $N$. A useful result in this direction is given by the following.
\begin{theorem}\label{t002}
Let $G,N$ be two groups such that $|G|=|N|$. Let $\mathfrak{f}\in\ha(G,\au(N))$ and $\mathfrak{g}\in Z_\mathfrak{f}^1$ be a bijective crossed homomorphism (i.e. $(G,N)$ is realizable). Then if $M$ is a characteristic subgroup of $N$ and $H=\mathfrak{g}^{-1}(M)$, we have
that the pair $(H,M)$ is realizable.
\end{theorem}
In recent work realizability of Cyclic groups has been characterized by the following two results. We mention them here as this will be the key ingredient to classify the same in case of Dihedral groups.
\begin{proposition}\cite[Theorem 3.1]{t}\label{p003}
Let $N$ be a group of odd order $n$ such that the pair $(\Z_n,N)$ is realizable. Then $N$ is a $C$-group.
\end{proposition}
\begin{proposition}\cite[Theorem 1]{r}\label{p004}
Let $G$ be a group of order $n$ such that $(G,\Z_n)$ is realizable. Then $G$ is solvable and almost Sylow-cyclic.
\end{proposition}

\section{Main Results}
\begin{proposition}\label{p005}
Let $N=\D_{2n}$ and $(G,N)$ is realizable. Then $G$ is solvable.
\end{proposition}
\begin{proof}
The non-trivial proper characteristic subgroups of $\D_{2n}=\langle r,s:r^n=s^2=srsr=1\rangle$ are given by
\begin{align*}
    \langle r^d\rangle, \text{ where }d|n.
\end{align*}
Then for all $\alpha|n$ the group $G$ has a subgroup of order $\alpha$, say $H_\alpha$ such that $(H_\alpha,\langle r^{\frac{n}{\alpha}}\rangle)$ is realizable. Since $G$ has order $2n$ and $H_n$ is a subgroup of index $2$, then $H_n$ is a normal subgroup of $G$. Let $n=p_1^{\beta_1}p_2^{\beta_2}\ldots p_k^{\beta_k}$. Now consider the series
\begin{align*}
    G\geq H_n\geq H_{n/p_1}\geq H_{n/p_1^2}\geq\ldots\geq H_{n/p_1^{\beta_1}p_2^{\beta_2}\ldots p_k^{\beta_k-1}}\geq 1.
\end{align*}
Then each of $H_\alpha/H_{\alpha/p_j}$ is cyclic, hence abelian. 
\end{proof}

\begin{theorem}\label{t001}
Let $N$ be a group of order $2n$, where $n$ is odd and the pair $(\Z_n\rtimes\Z_2,N)$ is realizable. Then $N \cong (\Z_k\rtimes \Z_l)\rtimes \Z_2$ where $(k,l)=1,lk=n$.  
\end{theorem}
\begin{proof}
By Proposition \ref{p001}, $N$ has a unique and hence characteristic subgroup $H_n$ of order $n$. Then by Proposition \ref{p002} there exists a bijective crossed homomorphism $\ga\in Z^1_{\fa}(\Z_n\rtimes\Z_2,N)$ for some $\fa\in \ha(\Z_n\rtimes\Z_2,\au(N))$.
Hence by Theorem \ref{t002} the pair $(\ga^{-1}H_n,H_n)$ is realizable. Note that $\Z_n\rtimes\Z_2$ has unique subgroup of order $n$, which is cyclic.
It follows that $\ga^{-1}H_n=\Z_n$. This implies that $(\Z_n,H_n)$ is realizable. Hence by Proposition \ref{p003} we get that $H_n$ is a $C$-group, whence it follows from \cite{bu} that $H_n=\Z_k\rtimes\Z_l$ for $(k,l)=1,kl=n$.
\end{proof}
\begin{theorem}\label{t003}
Let $G$ be a group of order $2n$ such that the pair $(G,\Z_n\rtimes\Z_2)$ is realizable. Then $G=(\Z_k\rtimes\Z_l)\rtimes\Z_l$ for some $(k,l)=1,kl=n$.
\end{theorem}
\begin{proof}
Given that the pair $(G,\Z_n\rtimes\Z_2)$ is realizable, by Proposition \ref{p002} there exists a bijective 
crossed homomorphism $\ga\in Z^1_{\fa}(G,\Z_n\rtimes\Z_2)$ for some $\fa\in \ha(G,\au(\Z_n\rtimes\Z_2))$. 
Since $\Z_n$ is a characteristic subgroups of $\Z_n\rtimes\Z_2$, we get that $\ga^{-1}(\Z_n)$ is a subgroup 
of $G$ and $(\ga^{-1}(\Z_n),\Z_n)$ is realizable. Then by Proposition \ref{p004}, we have that 
$\ga^{-1}(\Z_n)$ is almost Sylow-cylic. Hence by \cite{bu} $\ga^{-1}(\Z_n)=\Z_k\rtimes\Z_l$. Hence the result follows.
\end{proof}
\begin{corollary}\label{t004}
Let $n$ be an odd number such that $\ra(n)$ is a Burnside number. Assume that $|G|=|N|=2n$ and $(G,N)$ is realizable. Then $G=\Z_n\rtimes_{\varphi}\Z_2$ if and only if $N=\Z_n\rtimes_{\psi}\Z_2$. 
\end{corollary}
\begin{proof}
Let $r_1,r_2\in\mathbb{N}$ be two numbers such that $(r_1,r_2)=1$ and $\ra(r_1r_2)$ is a Burnside number. Take $i\neq j$ and $i,j\in\{1,2\}$. Then the group homomorphism
\begin{align*}
    \psi:\Z_{r_i}\rightarrow\au(\Z_{r_j})
\end{align*}
is trivial, which shows that $\Z_{r_i}\rtimes\Z_{r_j}=\Z_{r_i}\oplus\Z_{r_j}=\Z_{r_1r_2}$. Hence the result follows from previous two Theorems.
\end{proof}
\begin{remark}\label{r002}
It was shown in \cite{ap}, that if $G=\D_{2n}$ then for any $\phi\in\ha(\Z_2,\au(\Z_n))$ the pair 
$(\D_{2n},\Z_n\rtimes_{\phi}\Z_2)$ is realizable. This along with 
the previous theorem implies that when $\ra(n)$ is a Burnside 
number, these are the all possible realizable pairs. 
\end{remark}
\begin{corollary}\label{c002}
Let $L/K$ be a finite Galois extension with Galois group isomorphic to 
$\D_{2n}$ where $n$ is odd such that $\ra(n)$ is a Burnside number. Then the
number of Hopf-Galois structures on $L/K$ is
\begin{align*}
e(\D_{2n})=\displaystyle{\sum\limits_{m=0}^n2^m\chi(n-m)},
\end{align*}
where $\chi(w)$ is the coefficient of $x^w$ in the polynomial $\prod\limits_{p_u\in\pi(n)} (x+p_u^{\alpha_u})$.
\end{corollary}
\begin{proof}
This follows from Theorem \ref{t004} and \cite[Corollary 4.1]{ap}. Then the result follows from Theorem \ref{t004} and Remark \ref{r002}.
\end{proof}
\begin{theorem}
Let $n\in\mathbb{N}$ such that $n\equiv2\pmod{4}$. Suppose $(G,\D_{2n})$ is realizable. Then there exists a short exact sequence
\begin{align*}
    0\longrightarrow\Z_l\rtimes\Z_k\longrightarrow G\longrightarrow\Z_2\longrightarrow0,
\end{align*}
for some $(k,l)=1,kl=2n$.
\end{theorem}
\begin{proof}
Given that the pair $(G,\D_{2n})$ is realizable, we get that there exists $\fa\in\ha(G,\au(\D_{2n}))$ and $\ga\in Z_\fa^1(G,\D_{2n})$ 
corresponding to the regular embedding of $G$ in $\h(\D_{2n})$. Take $M=\langle r\rangle\subseteq \D_{2n}$, which is a characteristic subgroup of $\D_{2n}$.
Then the pair $(\ga^{-1}M,\Z_n)$ is realizable. Hence we have that $\ga^{-1}M\cong \Z_l\rtimes\Z_k$ for some $(k,l)=1,kl=2n$.
Since $\ga^{-1}M$ is a subgroup of index $2$, the result follows. 
\end{proof}

\end{document}